\documentclass{scrartcl}
\usepackage[utf8]{inputenc}
\usepackage{hyperref}
\usepackage[english]{babel}
\usepackage{enumerate}
\usepackage{aliascnt}
\usepackage{amssymb}
\usepackage{amsmath}
\usepackage{amsthm}
\usepackage{verbatim}
\usepackage{mathtools}
\usepackage{enumitem}

\usepackage{esint} 

\usepackage{tikz}
\usepackage{mathscinet} 

\usepackage[nottoc]{tocbibind} 
\usepackage[title]{appendix}

\usepackage{color}

\numberwithin{equation}{section} 

\usepackage{dsfont} 
\usepackage{hyperref} 
\usepackage[noabbrev]{cleveref}

\usepackage{autonum} 

\allowdisplaybreaks[4] 

\newtheorem{prop}{Proposition}[section]

\newaliascnt{lem}{prop} 

\newtheorem{lem}[lem]{Lemma}

\aliascntresetthe{lem} 
\Crefname{lem}{Lemma}{Lemmas}

\newaliascnt{defi}{prop} 
 \newtheorem{defi}[defi]{Definition}

\aliascntresetthe{defi} 
\Crefname{defi}{Definition}{Definitions}

\newaliascnt{cor}{prop} 
 
\aliascntresetthe{cor}

\newaliascnt{remark}{prop} 
 \newtheorem{remark}[remark]{Remark}
 
\aliascntresetthe{remark} 

\newaliascnt{thm}{prop} 
 \newtheorem{thm}[thm]{Theorem}
 
\aliascntresetthe{thm} 

\newaliascnt{example}{prop} 
 
\aliascntresetthe{example}

\def\equationautorefname~#1\null{%
  (#1)\null
}

\newcommand{\R}{\ensuremath{\mathbb{R}}}
\newcommand{\N}{\ensuremath{\mathbb{N}}}

\newcommand{\CalE}{\ensuremath{\mathcal{E}}}
\newcommand{\CalF}{\ensuremath{\mathcal{F}}}

\newcommand{\CalA}{\ensuremath{\mathcal{A}}}

\newcommand{\CalK}{\ensuremath{\mathcal{K}}}
\newcommand{\CalU}{\ensuremath{\mathcal{U}}}
\newcommand{\CalO}{\ensuremath{\mathcal{O}}}

\DeclareMathOperator{\diverg}{div}
\DeclareMathOperator{\curl}{curl}

\DeclareMathOperator{\CalV}{\ensuremath{\mathcal{V}}}

\DeclareMathOperator{\supp}{supp}

\DeclareMathOperator{\cl}{cl}
\DeclareMathOperator{\dist}{dist}

\title{Vanishing of long time average p-enstrophy dissipation rate in the inviscid limit of the 2D damped Navier-Stokes equations}
\author{Raphael Wagner\thanks{Institute of Applied Analysis, Ulm University, Helmholtzstra\ss e 18, 89081 Ulm, Germany. \texttt{raphael.wagner@uni-ulm.de}}}

\begin{document}

\maketitle
\begin{abstract}
\noindent\textbf{Abstract:} In \cite{CR07}, Constantin and Ramos proved a result on the vanishing long time average enstrophy dissipation rate in the inviscid limit of the 2D damped Navier-Stokes equations. In this work, we prove a generalization of this for the $p$-enstrophy, sequences of distributions of initial data and sequences of strongly converging right-hand sides. We simplify their approach by working with invariant measures on the global attractors which can be characterized via bounded complete solution trajectories. Then, working on the level of trajectories allows us to directly employ some recent results on strong convergence of the vorticity in the inviscid limit.
\end{abstract}


\section{Introduction and main results}
We consider the two-dimensional damped Navier-Stokes equations
\begin{equation}\label{eq: damped NSE vel}
\begin{split}
\partial_t u^\nu + (u^\nu\cdot\nabla)u^\nu + \gamma u^\nu + \nabla p^\nu - \nu\Delta u^\nu &=  f,\\
u^\nu(0) &= u_0
\end{split}
\end{equation}
in velocity $u^\nu = (u_1^\nu,u_2^\nu)\colon \R^2 \times [0,M) \to \R^2$ with constants $\nu, \gamma > 0$. Then $\omega^\nu \coloneqq \curl u^\nu = \partial_{x_1}u^\nu_2 - \partial_{x_2}u^\nu_1$ is called the vorticity of $u^\nu$ and formally satisfies
\begin{equation}\label{eq: damped NSE}
\begin{split}
\partial_t \omega^\nu + u^\nu\cdot\nabla\omega^\nu + \gamma \omega^\nu - \nu\Delta\omega^\nu &= g,\\
u^\nu &= K*\omega^\nu,
\end{split}
\end{equation} 
where $g \coloneqq \curl f$ and $K(x) \coloneqq \frac{1}{2\pi}\nabla^{\perp}\log|x|, x \in \R^2$.\\
The square of the $L^2(\R^2)$ norm of the vorticity $\omega^\nu(t)$ at some time $t$ is called enstrophy and by testing \eqref{eq: damped NSE} with $\omega^\nu(t)$ in the undamped, unforced situation ($\gamma = 0$ and $g = 0$), one obtains
\[\frac{d}{dt}\|\omega^\nu(t)\|^2_{L^2} + 2\nu\|\nabla \omega^\nu(t)\|_{L^2}^2\,dx = 0\]
so that $-2\nu\|\nabla \omega^\nu(t)\|_{L^2}^2$ represents the enstrophy dissipation rate at time $t$.\\
The inviscid analogue ($\nu = 0$) of \eqref{eq: damped NSE vel}, \eqref{eq: damped NSE}, i.e., the Euler equations, formally conserve the enstrophy and one would particularly expect the enstrophy dissipation rate to vanish in the inviscid limit $(\nu \to 0)$ of the Navier-Stokes equations.\\
For $1 \leq p < \infty$, conservation of $p$-enstrophy, by which we mean the $p$-th power of the $L^p(\R^2)$ norm of the vorticity, is a property of renormalized solutions to the vorticity formulation of the Euler equations (see \Cref{def: renorm sol} for a precise definition in our setting). In the class of bounded vorticity in $L^p(\R^2)$, $p \geq 2$, it is well-known that all weak vorticity solutions to the Euler equations are automatically renormalized \cite{DPL89,LFMNL06}. For $p < 2$, such renormalized solutions can be obtained as weak-* limits in $L^\infty((0,M);L^p(\R^2))$ of subsequences of (smooth) exact solutions to the Euler equations or the Navier-Stokes equations with vanishing viscosity \cite{LFMNL06,CS15,CNSS17}.\\
As to the question of vanishing $p$-enstrophy dissipation rate in the inviscid limit, $1\leq p < \infty$, this can at least be deduced for subsequences as a necessary consequence of the strong convergence in $C([0,M];L^p(\R^2))$ of a subsequence of $(\omega^\nu)_{\nu > 0}$ as $(\nu \to 0)$, which is a quite recent improvement in \cite{NLSW21,CCS21} of the aforementioned weak-* convergence.\\
In the above mentioned references, the final time $M$ is fixed. Here, however, likewise to \cite{CR07} and the related work \cite{CTV13} on the surface quasi-geostrophic equations, we do not study the inviscid limit on a finite time interval, but study it in the stationary situation that is obtained by first considering long times averages. More precisely, in \cite{CR07}, Constantin and Ramos proved that the long time average of the enstrophy dissipation rate of the 2D damped Navier-Stokes equations vanishes in the inviscid limit, i.e.,
\begin{equation}\label{eq: Constantin Ramos}
\lim_{\nu\to 0} \nu \limsup_{M\to\infty} \frac{1}{M}\int_0^M\int_{\R^2} |\nabla \omega^\nu|^2\,dx\,dt = 0,
\end{equation}
under the assumptions $f, u_0 \in W^{1,1}(\R^2) \cap W^{1,\infty}(\R^2)$.\\
The main result of this article, which is stated in the following theorem, generalizes this for the $p$-enstrophy, $1 < p < \infty$. By $\lbrace S^\nu(t)\rbrace_{t \geq 0}$, we denote the solution semigroup of the damped Navier-Stokes equations associated to a right-hand side $f^\nu$ on the space $\CalE_1 \cap \CalE_p$ of weakly divergence-free vector fields in $L^2(\R^2)$ whose curl is in $(L^1 \cap L^p)(\R^2)$. In particular, rather than considering fixed $f \in W^{1,1}(\R^2) \cap W^{1,\infty}(\R^2)$, we consider bounded $(f^\nu)_{\nu > 0}$ in $\CalE_1 \cap \CalE_p$ s.t. $(\curl f^\nu)_{\nu > 0}$ is precompact in $L^r(\R^2)$ with $r = \max\lbrace 2,p\rbrace$. Moreover, we allow for sequences of distributions of initial data having finite $p$-th moment w.r.t. the $L^p(\R^2)$ norm, rather than fixed $u_0 \in W^{1,1}(\R^2) \cap W^{1,\infty}(\R^2)$.

\begin{thm}\label{thm: main thm}
Let $(\mu_0^{\nu})_{\nu > 0}$ be a family of Borel probability measures on $\CalE_1 \cap \CalE_p$, satisfying
\begin{equation}\label{eq: main thm finite p moment}
\int_{\CalE_1 \cap \CalE_p} \|\omega(u_0)\|_{L^p}^p\,d\mu^{\nu}_0(u_0) < \infty
\end{equation}
for every $\nu > 0$. Then,
\begin{equation}\label{eq: vanishing mean enstrophy}
\lim_{\nu \to 0}\nu \limsup_{M \to \infty} \frac{1}{M}\int_0^M\int_{\CalE_1 \cap \CalE_p} \int_{\R^2}|\nabla|\omega(S^\nu(t) u_0)|^\frac{p}{2}|^2\,dx\,d\mu_0^\nu(u_0)\,dt = 0.
\end{equation}
\end{thm}

More information on the definitions and assumptions will be given in \Cref{sec: weak and renormed sols} and \Cref{sec: vanishing p-enstrophy dissip rate}.\\
The proof of this main theorem follows the same approach as \cite{CR07} by first considering measures associated to long time averages for fixed $\nu > 0$ and to then pass on to the limit $(\nu \to 0)$. However, we will employ that these measures are necessarily concentrated on the global attractor of the solution semigroup of the damped Navier-Stokes equations and use its characterization via bounded, complete solution trajectories, which allows us to work on the level of trajectories subsequently. In contrast, Constantin and Ramos work with what are called statistical solutions in phase space. This leads to a notion of renormalized stationary statistical solutions and to the adaptation of some more or less technical arguments of the seminal work \cite{DPL89} on renormalization theory by DiPerna and Lions to the context of phase space statistical solutions. Also, stronger assumptions on the right-hand side were made for these constructions as indicated above.\\
Moreover, by working on the level of trajectories, we may employ some recent results on the inviscid limit of solutions to the Navier-Stokes equations. In particular, the improvement to convergence in $C([0,M];L^p(\R^2))$ in the inviscid limit to renormalized solutions of the corresponding damped Euler equations \cite{CCS21,NLSW21}, which we already mentioned above, will be of use.\\
After this introduction, in \Cref{sec: weak and renormed sols}, we briefly recall some classical and recent theory on renormalized solutions of the damped Euler and Navier-Stokes equations and the inviscid limit. We continue on in \Cref{sec: attractors} with a discussion on the global attractor of the damped Navier-Stokes equations based on \cite{CIZ17,CVZ11,IPZ15,IC17}. Next, in \Cref{sec: time-average invar meas}, we glimpse over a Krylov-Bogolioubov type construction of invariant measures that are concentrated on the global attractor. In particular, we rely on work in \cite{LRR11}. Finally, in \Cref{sec: vanishing p-enstrophy dissip rate}, we prove the main theorem.

\section{Weak and renormalized solutions of the damped Euler and Navier-Stokes equations}\label{sec: weak and renormed sols}

For $1 \leq p \leq \infty$ and $m \in \N_0$, we denote by $L^p(\R^2)$ and $H^m(\R^2)$ the standard $L^p$ spaces and the $L^2$ based Sobolev spaces of functions on $\R^2$ with values in $\R$ or $\R^2$, depending on the context.\\
We then denote by $H$ the subspace of all velocity fields in $L^2(\R^2)$ which are weakly divergence-free, define $H^m \coloneqq H^m(\R^2) \cap H$ and denote its dual space by $H^{-m}$.\\
Then, we may introduce the following sets
\[\CalE_p \coloneqq \lbrace F = (F_1,F_2) \in L^2(\R^2) : \curl F \coloneqq \partial_{x_1}F_2 - \partial_{x_2}F_1 \in L^p(\R^2)\rbrace\]
for every $1 \leq p \leq \infty$. Note that $\CalE_p$ is a Banach space with norm $\|F\|_{\CalE_p} \coloneqq \|F\|_{L^2} + \|\curl F\|_{L^p}$, $F \in \CalE_p$.\\
We will frequently impose that some vector field $F$ is in $\CalE_1 \cap \CalE_p$ throughout this work. The assumption of being in $\CalE_1$ is usually made as it guarantees that $F$ is uniquely determined by its vorticity via the Biot-Savart law (the second equation in \eqref{eq: damped NSE}).\\
Finally, let us also mention here that the subscript $loc$ in our notation for spaces of continuous functions means that we endow this space with the topology of uniform convergence on compact intervals, while in the notation for Lebesgue spaces, this denotes integrability and the topology of convergence on compact subsets.
\hfill\\
Now, throughout this section, we fix $p \in (1,\infty)$ and assume, unless stated otherwise, for the right-hand side in \eqref{eq: damped NSE vel}, 
\begin{equation}\label{eq: rhs}
f \in \CalE_1 \cap \CalE_p \text{ and define } g \coloneqq \curl f.
\end{equation}
This particularly means that $f$ and $g$ do not depend on time. Also, for a given velocity field $u \in \CalE_p$, we will usually denote its vorticity by
\[\omega(u) \coloneqq \curl u \in L^p(\R^2).\]
Since we are ultimately interested in the inviscid limit $(\nu \to 0)$, we may assume throughout that $0 < \nu < 1$ and $\nu$ is fixed unless stated otherwise.\\
The inviscid analogue of \eqref{eq: damped NSE} generally cannot be interpreted in the weak sense if $p < \frac{4}{3}$ and other types of solution concepts have emerged in this case. The frequently used notion of renormalized solutions is the one we consider here and define next.

\begin{defi}\label{def: renorm sol}
Let $\omega_0 \in (L^1\cap L^p)(\R^2)$ and $\omega \in C_{loc}([0,\infty);(L^1 \cap L^p)(\R^2))$. Then $\omega$ is called a renormalized solution of the vorticity formulation of the damped Euler equations with initial data $\omega_0$ and right-hand side $g$ if for every $\beta \in C^1(\R^2)$, bounded with bounded derivative and vanishing in a neighbourhood of 0,
\[\int_0^\infty\int_{\R^2} \beta(\omega)(\partial_t\varphi + u\cdot\nabla\varphi)-\gamma\omega\beta'(\omega)\varphi\,dx\,dt + \int_{\R^2} \beta(\omega_0)\varphi(0)\,dx = -\int_0^\infty\int_{\R^2}\beta'(\omega)g\varphi\,dx\,dt,\]
where $\varphi \in C^\infty_c([0,\infty)\times \R^2)$ and 
\[u(t,x) = (K * \omega(t))(x) \text{ for } a.e.\, (t,x) \in (0,\infty)\times\R^2.\]
\end{defi}

\begin{remark}\label{rem: renormalized solutions}
\begin{enumerate}
\item[i)] One can also define a similar notion of  renormalized solutions of the vorticity formulation of the damped Navier-Stokes equations. This is unnecessary, however, since the notions of weak and renormalized solutions in this setting coincide as a well-known consequence of the consistency result in \cite{DPL89}.\\
Likewise, for $p \geq 2$, the notions of weak and renormalized solutions of the vorticity formulation of the damped Euler equations coincide.
\item[ii)] The classical DiPerna-Lions article \cite{DPL89} generally allows for a damping term. Most references related to the Euler equations to which we refer do not include our constant coefficient term and we remark here in general that these results may be proved analogously with some minor adaptations.
\item[iii)] It was proved in \cite{CS15} and \cite{CNSS17} that renormalized solutions to the vorticity formulation of the Euler equations can be constructed as weak-* limits in $L([0,\infty);L^p(\R^2))$ of (subsequences of) solutions of the vorticity formulation of the Navier-Stokes equations in the inviscid limit. We also state this fact in \Cref{thm: weak convergence in lp}. The weak solutions of the Navier-Stokes equations on the other hand can be constructed by standard methods such as Galerkin approximations or approximations by solving related smoothed equations from which, in either case, the (unique) weak solution to the Cauchy problem inherits several properties that we state in \Cref{thm: NSE existence and uniqueness}.
\end{enumerate}
\end{remark}

\begin{thm}\label{thm: NSE existence and uniqueness}
Let $u_0 \in \CalE_1 \cap \CalE_p$. Then there exists a unique weak solution $u \in C_{loc}([0,\infty);\CalE_1 \cap \CalE_p)$ of the damped Navier-Stokes equations with initial data $u_0$ and right-hand side $f$ as in \eqref{eq: rhs}. Moreover,
\begin{itemize}[label={-}]
\item $\omega(u) \in C_{loc}(\R;(L^1 \cap L^p)(\R^2))$ is a weak solution of the vorticity formulation of the damped Navier-Stokes equations with initial data $\omega(u_0)$ and right-hand side $g$,
\item for all $M > 0$ and some $L > 0$ (independently of $0 < \nu < 1$!),
\begin{equation}\label{eq: NSE bdd time-derivative}
\|\partial_t u\|_{L^\infty(0,M;H^{-L}_{loc})} \leq C(M,\gamma,\|u_0\|_{L^2},\|f\|_{L^2}),
\end{equation}
\item a.e. on $[0,\infty)$, with equality for $p \geq 2$,
\begin{equation}\label{eq: NSE energy equ}
\begin{split}
&\frac{d}{dt} \int_{\R^2} |\omega(u)|^p\,dx + \gamma p \int_{\R^2}|\omega(u)|^p\,dx\\ 
&\qquad\leq -4\nu \frac{p-1}{p}\int_{\R^2} |\nabla |\omega(u)|^\frac{p}{2}|^2\,dx + p\int_{\R^2}g\omega(u)|\omega(u)|^{p-2}\,dx,
\end{split}
\end{equation}
\item for all $M > 0$,
\begin{equation}\label{eq: NSE energy ineq}
\begin{split}
&\|u\|_{L^\infty(0,M;L^2)} + \nu\|\nabla u\|_{L^2(0,M;L^2)} \leq C(\gamma,\|u_0\|_{L^2},M^\frac{1}{2}\|f\|_{L^2}),\\
&\|\omega(u)\|_{L^\infty(0,M;L^p)} + \nu\|\nabla|\omega(u)|^\frac{p}{2}\|_{L^2(0,M;L^2)}^\frac{2}{p} \leq C( \gamma,\|\omega(u_0)\|_{L^p},M^\frac{1}{p}\|g\|_{L^p}),
\end{split}
\end{equation}
\item for all $t \geq 0$ and $1 \leq q \leq p$ (independently of $0 < \nu < 1$!),
\begin{equation}\label{eq: NSE vort lp bound}
\begin{split}
&\|u(t)\|_{L^2} \leq e^{-\gamma t}(\|u_0\|_{L^2} - \frac{1}{\gamma}\|f\|_{L^2}) + \frac{1}{\gamma}\|f\|_{L^2},\\
&\|\omega(u)(t)\|_{L^q} \leq e^{-\gamma t}(\|\omega(u_0)\|_{L^q} - \frac{1}{\gamma}\|g\|_{L^q}) + \frac{1}{\gamma}\|g\|_{L^q}.
\end{split}
\end{equation}
\end{itemize}
\end{thm}

\begin{thm}\label{thm: EE existence}
Let $\omega \in C_{loc}([0,\infty);(L^1 \cap L^p)(\R^2))$ be a renormalized solution of the vorticity formulation of the damped Euler equations. Then
\begin{equation}\label{eq: EE energy equ}
\frac{d}{dt} \int_{\R^2} |\omega|^p\,dx + \gamma p \int_{\R^2}|\omega|^p\,dx = p\int_{\R^2}g\omega|\omega|^{p-2}\,dx
\end{equation}
a.e. on $[0,\infty)$.
\end{thm}

The weak-* convergence in $L^\infty([0,\infty);L^p(\R^2))$ of weak vorticity solutions of the Navier-Stokes equations in the inviscid limit $(\nu \to 0)$ to renormalized solutions of the Euler equations was improved recently in \cite{NLSW21,CCS21}, which we state in the following theorem. We note that this remains true in our case of an additional constant coefficient damping term, strongly converging right-hand sides and in case one considers viscosity $\nu$ which converges to some positive number, i.e., better behaved convergence of Navier-Stokes solutions to Navier-Stokes solutions.

\begin{thm}\label{thm: strong convergence in lp}
Let $(\nu^k)_{k\in\N} \subset (0,1)$ be a sequence converging to some $\nu \geq 0$ and let $(u_0^{\nu^k})_{k\in\N}, (f^{\nu^k})_{k\in\N}$ be bounded sequences of initial data and right-hand sides in $\CalE_1 \cap \CalE_p$ with associated weak solutions $(u^{\nu^k})_{k\in\N}$ as in \Cref{thm: NSE existence and uniqueness}, where $\nu^k$ is the viscosity constant for every $k \in \N$. We suppose that for some $u_0^\nu,f^\nu \in \CalE_1 \cap \CalE_p$,
\[
\omega(u_0^{\nu^k}) \to \omega(u_0^\nu)\,(k \to \infty) \text{ in } L^p(\R^2)
\]
and for $g^{\nu^k} \coloneqq \curl f^{\nu^k}$, $g^\nu \coloneqq \curl f^\nu$
\[
g^{\nu^k} \to g^\nu\,(k \to \infty) \text{ in } L^p(\R^2).
\]
Then, after passing to a subsequence, there exists $u^\nu$ such that
\begin{align}
&u^{\nu^k} \overset{\ast}{\rightharpoonup} u^\nu\,(k \to \infty) \text{ in } L^\infty([0,\infty);L^2(\R^2)),\\
&\omega(u^{\nu^k}) \to \omega(u^\nu) \,(k \to \infty)\text{ in } C_{loc}([0,\infty);L^p(\R^2)),
\end{align}
where $\omega(u^\nu)$ satisfies the vorticity formulation of the damped Euler $(\nu = 0)$ or Navier-Stokes $(\nu > 0)$ equations (in the renormalized sense) with initial data $\omega(u_0^\nu)$ and right-hand side $g^\nu$.
\end{thm} 

We also state the related main theorem from \cite{CS15}, adapted to our situation. Although it appears to be weaker, we give the theorem in a slightly improved and very useful form of only requiring weak convergence of the initial data opposed to strong convergence as originally stated in \cite{CS15}, which we explain below.

\begin{thm}\label{thm: weak convergence in lp}
Let $(\nu^k)_{k\in\N} \subset (0,1)$ be a sequence converging to $0$ and let $(u_0^{\nu^k})_{k\in\N},(f^{\nu^k})_{k\in\N}$ be bounded sequences of initial data and right-hand sides in $\CalE_1 \cap \CalE_p$ with associated weak solutions $(u^{\nu^k})_{k\in\N}$ as in \Cref{thm: NSE existence and uniqueness}, where $\nu^k$ is the viscosity constant for every $k \in \N$. We suppose that for some $u_0,f \in \CalE_1 \cap\CalE_p$,
\[
\omega(u_0^{\nu^k}) \rightharpoonup \omega(u_0)\,(k \to \infty) \text{ in } L^p(\R^2)
\]
and for $g^{\nu^k} \coloneqq \curl f^{\nu^k}$, $g \coloneqq \curl f$,
\[
g^{\nu^k} \to g\,(k \to \infty) \text{ in } L^p(\R^2).
\]
Then, after passing to a subsequence, there exists $u$ such that
\begin{align}
&u^{\nu^k} \overset{\ast}{\rightharpoonup} u\,(k \to \infty) \text{ in } L^\infty([0,\infty);L^2(\R^2)),\\
&\omega(u^{\nu^k}) \overset{\ast}{\rightharpoonup} \omega(u)\,(k \to \infty) \text{ in } L^\infty([0,\infty);L^p(\R^2)),
\end{align}
where $\omega(u)$ satisfies the vorticity formulation of the damped Euler equations in the renormalized sense with initial data $\omega(u_0)$ and right-hand side $g$.
\end{thm} 

In the proof given in \cite{CS15}, strong convergence of the initial data is only required in (3.10) to pass to the limit in the integral 
\begin{equation}\label{eq: init data}
\int_{\R^2} \phi^{\nu}(x,0)\omega(u_0^{\nu})(x)\,dx,
\end{equation}
where $\phi^\nu$ is the unique solution to the backward transport-diffusion problem 
\begin{equation}
\begin{split}
-\partial_t\phi^\nu -\nu\Delta\phi^\nu - \diverg(\phi^\nu u^\nu) + \gamma \phi^\nu &= \chi\\
\phi^\nu(x,M) &= 0
\end{split}
\end{equation}
with right-hand side $\chi \in C_c^\infty(\R^2\times (0,M))$ for some arbitrary but fixed $M > 0$. The functions $(\phi^\nu)_{\nu > 0}$ converge in $C([0,M];L^q_w(\R^2))$ ($q = \frac{p}{p-1}$) to $\phi$ which solves the corresponding backward transport equation with $u^\nu$ replaced by $u$. As the final value at time $M$ is fixed to $0$, Proposition 1 in \cite{NLSW21} or Lemma 3.3 in \cite{CCS21}, which are used to prove our here formulated \Cref{thm: strong convergence in lp}, can be altered to obtain strong convergence of $(\phi^\nu)_{\nu > 0}$ in $C([0,M];L^q(\R^2))$ to $\phi$. Then, in order to pass to the limit in \eqref{eq: init data}, weak $L^p(\R^2)$ convergence of $(\omega(u_0^\nu))_{\nu > 0}$ actually suffices.

Finally, let us close this section with a result on the vorticity of solutions to the Navier-Stokes equations being instantaneously in $L^2(\R^2)$ with a bound on the corresponding norm for positive times bounded away from $0$, even for initial data of lower integrability. Our estimates are derived as in the proof of Theorem B in \cite{A94} but with the inclusion of a damping term and a right-hand side $g \in (L^1 \cap L^2)(\R^2)$. 

\begin{lem}\label{lem: NSE bdd l2 by l1}
Consider the unique solution $u \in C_{loc}([0,\infty);\CalE_1 \cap \CalE_p)$ of the damped Navier-Stokes equations with initial data $u_0 \in \CalE_1 \cap \CalE_p$, right-hand side $f \in \CalE_1 \cap \CalE_{\max\lbrace 2,p\rbrace}$ and define $g \coloneqq \curl f$. Then, for fixed $t_0 > 0$, $\|\omega(u)(t)\|_{L^2}$ is bounded in $t \geq t_0$ by a constant $C = C(t_0,\gamma,\nu,\|g\|_{L^1 \cap L^2}, \|\omega(u_0)\|_{L^1 \cap L^p})$. 
\end{lem}

\begin{proof}
As the unique weak solution in \Cref{lem: NSE bdd l2 by l1} can be constructed by approximation with smooth solutions having smooth, compactly supported initial data, it suffices to show the lemma under these regularity assumptions.\\
Equation \eqref{eq: NSE energy equ} yields
\begin{equation}\label{eq: NSE bdd l2 by l1 (1)}
\frac{d}{dt} \|\omega(u)(t)\|_{L^2}^2 \leq -2\nu \|\nabla\omega(u)(t)\|_{L^2}^2 + \frac{\|g\|_{L^2}^2}{\gamma}
\end{equation}
for $t \geq 0$. The right-hand side in \eqref{eq: NSE vort lp bound} for $q = 1$ can be bounded by a constant $1 \leq c = c(\gamma,\|\omega(u_0)\|_{L^1},\|g\|_{L^1})$ so that the Nash inequality yields
\[\|\omega(u)(t)\|^2_{L^2} \leq \eta \|\omega(u)(t)\|_{L^1}\|\nabla \omega(u)(t)\|_{L^2} \leq c\eta\|\nabla \omega(u)(t)\|_{L^2},\]
for some constant $\eta > 0$. Then \eqref{eq: NSE bdd l2 by l1 (1)} yields
\[\frac{d}{dt}\|\omega(u)(t)\|_{L^2}^2 \leq -\frac{2\nu}{c^2\eta^2}\|\omega(u)(t)\|_{L^2}^4 + \frac{\|g\|_{L^2}^2}{\gamma}.\]
Therefore, to estimate $\|\omega(u)(t)\|_{L^2}^2$, we consider the corresponding Riccati differential equation
\begin{equation}\label{eq: riccati}
\dot{x} = -\frac{2\nu}{c^2\eta^2} x^2 + \frac{\|g\|_{L^2}^2}{\gamma}.
\end{equation}
We remark that generally for $a,b > 0$ and $x_0 > 0$ large enough such that $\sqrt{\frac{b}{a}}x_0 > 1$, the solution to 
\[\dot{x} = -ax^2 +b, x(0) = x_0\]
is given by 
\[x(t) = \sqrt{\frac{b}{a}}\coth(\sqrt{ab}(z+t))\]
for some $z > 0$, where $\coth(s) = \frac{e^{2s}+1}{e^{2s}-1}$ for $s > 0$. We can assume $\omega(u_0),g \neq 0$. As the lemma holds immediately if $\omega(u_0) \in L^2(\R^2)$ due to \eqref{eq: NSE vort lp bound}, we may also assume that we consider smooth, compactly supported approximations of initial vorticity in $L^1(\R^2) \setminus L^2(\R^2)$. Consequently, we suppose $\|\omega(u_0)\|_{L^2}$ to be large. Then, choosing $a$ and $b$ to match \eqref{eq: riccati}, $\sqrt{\frac{b}{a}}\|\omega(u_0)\|_{L^2} > 1$. As $\coth$ is bounded on intervals that are bounded away from $0$, this yields the claim of the lemma.
\end{proof}

\section{Attractors of the damped Navier-Stokes equations}\label{sec: attractors}
The global attractor of the solution semigroup of the damped Navier-Stokes equations on $H$, without assumptions on the vorticity, was constructed in \cite{IPZ15}. We recall the necessary definitions and slightly enhance the attraction and compactness property of the attractor for $L^p$ bounded vorticities. There, we will make an assumption on the right-hand side so that we may argue in a way that also applies to inviscid limit sequences. Except for these inviscid limit considerations, the viscosity $\nu > 0$ will be fixed again throughout this section.\\
All these arguments blend in with the work in \cite{CIZ17,CVZ11,IPZ15,IC17}.

\begin{defi}
Let $\lbrace \theta(t) \rbrace_{t \geq 0}$ be a semigroup on a complete metric space $X$. A set $\CalU \subset X$ is called global attractor of $\lbrace \theta(t) \rbrace_{t \geq 0}$ if 
\begin{enumerate}
\item $\CalU$ is compact in $X$,
\item $\CalU$ is attracting in the sense that for any bounded set $B \subset X$ and every open neighbourhood $\CalO$ of $\CalU$, there exists $t_0 = t_0(B,\CalO) > 0$ such that $S(t)B \subset \CalO$ for all $t \geq t_0$,
\item $\CalU$ is invariant, that is, $S(t)\CalU = \CalU$ for every $t \geq 0$.
\end{enumerate}
\end{defi}

\begin{thm}\cite{IPZ15}\label{thm: ex global attractor}
Let $\lbrace S(t) \rbrace_{t\geq 0}$ be the solution semigroup of the damped Navier-Stokes equations on $H$. Then $\lbrace S(t) \rbrace_{t \geq 0}$ has a global attractor $\CalA$.\\
Moreover, if we let $\CalK \subset C_b(\R;L^2(\R^2))$ be the set of complete solution trajectories of the damped Navier-Stokes equations that are bounded in $L^2(\R^2)$, then
\[\CalA = \lbrace u(0) : u \in \CalK\rbrace.\]
\end{thm}

We will continue to denote the solution semigroup of the damped Navier-Stokes equations on $H$ by $\lbrace S(t)\rbrace_{t \geq 0}$ and its attractor by $\CalA$.\\
The set $\CalK$ is bounded in $C_b(\R;H)$ as can be seen as follows. Let $u \in \CalK$ and $R > 0$ s.t. $\sup_{t\in\R}\|u(t)\|_{L^2} \leq R$. Then, for any $t\in\R$ and arbitrary $s > 0$, \eqref{eq: NSE vort lp bound} implies
\begin{equation}
\begin{split}
\|u(t)\|_{L^2} &\leq e^{-\gamma s}\left(\|u(t-s)\|_{L^2} - \frac{\|f\|_{L^2}}{\gamma}\right) + \frac{\|f\|_{L^2}}{\gamma}\\
&\leq e^{-\gamma s}\left(R - \frac{\|f\|_{L^2}}{\gamma}\right) + \frac{\|f\|_{L^2}}{\gamma}\\
&\to \frac{\|f\|_{L^2}}{\gamma}\,(s\to \infty)
\end{split}
\end{equation}
and we conclude
\[\sup_{t\in\R} \|u(t)\|_{L^2} \leq \frac{\|f\|_{L^2}}{\gamma}.\]
In preparation to the following proposition, we define
\[\CalK_p \coloneqq \lbrace u \in \CalK : \sup_{t \in \R} \|\omega(u)(t)\|_{L^q} \leq \frac{\|g\|_{L^q}}{\gamma}, q \in \lbrace 1,p \rbrace \rbrace.\]
Moreover, we make the assumption
\begin{equation}\label{eq: g}
f \in \CalE_1 \cap \CalE_{\max\lbrace 2,p\rbrace} \text{ and define } g \coloneqq \curl f.
\end{equation}

\begin{prop}\label{prop: attractor lp vort}
Let $B \subset \CalE_1 \cap \CalE_p$ be bounded. Under assumption \eqref{eq: g}, the set 
\[\CalA_p \coloneqq \lbrace u(0) : u \in \CalK_p\rbrace\]
is both compact and attracts $B$ w.r.t. $\|\cdot\|_{\CalE_p}$.
\end{prop}

\begin{proof}
We only show the attraction property. The compactness follows similarly.\\
Describing the attraction property in terms of sequences, this would mean that for an arbitrary bounded sequence $(u_0^n)_{n\in\N} \subset \CalE_1 \cap \CalE_p$ and a sequence of times $t^n \to \infty\,(n \to \infty)$, there exists a subsequence of $(S(t^n)u_0^n)_{n\in\N}$ converging to some $u_0 \in \CalA_p$ w.r.t. the norm of $\CalE_p$.\\
First of all, consider solutions of the damped Navier-Stokes equations $(u^n)_{n\in\N} \subset C_{loc}([-t^n,\infty);\CalE_1 \cap \CalE_p)$ satisfying the initial condition $u^n(-t^n) = u_0^n$ for every $n \in \N$. Then $S(t^n)u_0^n = u^n(0)$ for every $n \in \N$ and it suffices to show that $(u^n(0))_{n\in\N}$ converges w.r.t. $\CalE_p$ to an element of $\CalA_p$.\\
Since we already know that $\CalA$ is the attractor of $\lbrace S(t) \rbrace_{t \geq 0}$, we may assume that there exists $u \in \CalK$ such that for $u_0 \coloneqq u(0) \in \CalA$ and a subsequence, which we do not relabel, 
\[u^n(0) \to u_0\,(n\to\infty) \text{ in } L^2(\R^2).\]
As $(\omega(u_0^n))_{n\in\N}$ is bounded in $(L^1 \cap L^p)(\R^2)$, we may also assume due to the bound \eqref{eq: NSE vort lp bound},
\begin{equation}\label{eq: weak lp conv attractor}
\omega(u^n)(0) \rightharpoonup \omega(u_0)\,(n\to\infty) \text{ in } L^p(\R^2).
\end{equation} 
It now remains to show that $u \in \CalK_p$ and that we have strong convergence in \eqref{eq: weak lp conv attractor} for which it suffices to show convergence of the norms $\lim_{n \to \infty} \|\omega(u^n)(0)\|_{L^p} = \|\omega(u_0)\|_{L^p}$.\\
For this, we first note that we may naturally extend $u^n$ constantly by $u^n_0$ to $(-\infty,-t^n]$ for every $n \in \N$ so that $(u^n)_{n\in\N}$ is bounded in $C_b(\R;\CalE_1 \cap \CalE_p)$. In particular, we may assume
\begin{equation}\label{eq: weak conv vorticity}
\omega(u^n) \overset{\ast}{\rightharpoonup} \omega(u)\,(n \to \infty) \text{ in } L^\infty(\R;L^p(\R^2))
\end{equation}
and, due to \eqref{eq: NSE bdd time-derivative} and \eqref{eq: NSE energy ineq} along with the Sobolev inequality, $u^n \to u\,(n\to\infty)$ in $L^2_{loc}(\R;L^\frac{p}{p-1}_{loc}(\R^2))$. This yields convergence of the product terms $(u^n \omega(u^n))_{n\in\N}$ in the sense of distributions to $u\omega(u)$ so that $\omega(u)$ can be seen to also satisfy the vorticity formulation of the damped Navier-Stokes equations with initial data $\omega(u_0)$ on any compact time interval.\\
The bound on the vorticity in the definition of $\CalK_p$ can then be derived for $\omega(u)$ from \eqref{eq: NSE vort lp bound} so that $u \in \CalK_p$. Now, for every $n \in \N$,
\begin{equation}\label{eq: equ for initial data}
\begin{split}
&\|\omega(u^n)(0)\|_{L^p}^p- \|\omega(u^n)(-t^n)\|^p_{L^p}e^{-p\gamma t^n}\\
\leq&- 4\nu\frac{p-1}{p}\int_{-t^n}^0 e^{p\gamma t}\int_{\R^2}|\nabla|\omega(u^n)|^\frac{p}{2}|^2\,dx\,dt + p\int_{-t^n}^0 e^{p\gamma t}\int_{\R^2} g\omega(u^n)|\omega(u^n)|^{p-2}\,dx\,dt.
\end{split}
\end{equation}
As $\|\omega(u^n)(-t^n)\|_{L^p} = \|\omega(u^n_0)\|_{L^p}$ is bounded in $n \in \N$, it follows that $\|\omega(u^n)(-t^n)\|_{L^p}^p e^{-p\gamma t^n} \to 0\,(n \to \infty)$.\\
To treat the two terms on the right-hand side, we first differ between the cases $p = 2$ and $p > 2$. Under assumption \eqref{eq: g}, it is easy then to trace the case $1 < p < 2$ back to the case $p = 2$.\\
$p = 2$: From weak-* convergence \eqref{eq: weak conv vorticity}, that is, $\omega(u^n) \overset{\ast}{\rightharpoonup} \omega(u)\,(n\to\infty)$ in $L^\infty(\R;L^2(\R^2))$,
we obtain
\[2\int_{-t^n}^0 e^{2\gamma t}\int_{\R^2}g \omega(u^n) \,dx\,dt \to 2\int_{-\infty}^0e^{2\gamma t}\int_{\R^2} g\omega(u)\,dx\,dt\,(n\to\infty).\]
Due to \eqref{eq: NSE energy ineq}, we may also assume $\nabla \omega(u^n) \rightharpoonup \nabla \omega(u)\,(n\to\infty)$ in $L^2_{loc}(\R;L^2(\R^2))$. Using lower semi-continuity of the norm, we obtain for every $M>0$,
\begin{equation}
\begin{split}
&\liminf_{n \to \infty} \int_{-t^n}^0 e^{2\gamma t}\int_{\R^2}|\nabla|\omega(u^n)||^2\,dx\,dt = \liminf_{n \to \infty} \int_{-t^n}^0 e^{2\gamma t}\int_{\R^2}|\nabla\omega(u^n)|^2\,dx\,dt\\
\geq& \liminf_{n \to \infty}\int_{-M}^0 e^{2\gamma t}\int_{\R^2}|\nabla\omega(u^n)|^2\,dx\,dt
\geq \int_{-M}^0 e^{2\gamma t}\int_{\R^2}|\nabla\omega(u)|^2\,dx\,dt
= \int_{-M}^0 e^{2\gamma t}\int_{\R^2}|\nabla|\omega(u)||^2\,dx\,dt
\end{split}
\end{equation}
so that
\begin{equation}
\liminf_{n \to \infty} \int_{-t^n}^0 e^{2\gamma t}\int_{\R^2}|\nabla|\omega(u^n)||^2\,dx\,dt \geq \int_{-\infty}^0 e^{2\gamma t}\int_{\R^2}|\nabla|\omega(u)||^2\,dx\,dt.
\end{equation}
Combining these results yields
\[\limsup_{n \to \infty} \|\omega(u^n)(0)\|_{L^2} \leq -2\nu \int_{-\infty}^0 e^{2\gamma t}\int_{\R^2}|\nabla|\omega(u)||^2\,dx\,dt + 2\int_{-\infty}^0 e^{2\gamma t}\int_{\R^2}g \omega(u)\,dx\,dt.\]
On the other hand, the vorticity $\omega(u)$ of the complete trajectory $u$ satisfies $\eqref{eq: NSE energy equ}$ almost everywhere on $\R$ from which we obtain
\begin{equation}\label{eq: compact init data enstrophy balance}
\|\omega(u_0)\|_{L^2} = -2\nu \int_{-\infty}^0 e^{2\gamma t}\int_{\R^2}|\nabla|\omega(u)||^2\,dx\,dt + 2\int_{-\infty}^0 e^{2\gamma t}\int_{\R^2}g \omega(u)\,dx\,dt.
\end{equation}
Note that even though the right-hand side in \eqref{eq: NSE energy ineq} depends on the length of the time interval, the exponential function in the integrals above guarantees their convergence on $(-\infty,0]$.\\
We finally arrive at
\[\limsup_{n\to\infty} \|\omega(u^n)(0)\|_{L^2} \leq \|\omega(u_0)\|_{L^2} \leq \liminf_{n\to\infty}\|\omega(u^n)(0)\|_{L^2}.\]
$p > 2$: We note that as we have boundedness of the initial vorticities $(\omega(u^n_0))_{n\in\N}$ in $L^1(\R^2)$, by interpolation, we are also in the situation of the previous case $p = 2$, which means that next to weak-* convergence \eqref{eq: weak conv vorticity}, that is, $\omega(u^n) \overset{\ast}{\rightharpoonup} \omega(u)\,(n\to\infty)$ in $L^\infty(\R;L^p(\R^2))$, we may also assume $\omega(u^n)(0) \to \omega(u)(0)\,(n \to \infty)$ in $L^2(\R^2)$. By \Cref{thm: strong convergence in lp}, this yields strong convergence
\begin{equation}\label{eq: strong convergence on pos times}
\omega(u^n) \to \omega(u)\,(n \to \infty) \text{ in } C_{loc}([0,\infty);L^2(\R^2)).
\end{equation}
Note that we can use any other initial time $t_0 < 0$ in the above argument to obtain strong convergence on any compact subset of $\R$, not just of $[0,\infty)$ and \eqref{eq: strong convergence on pos times} holds in $C_{loc}(\R;L^2(\R^2))$.\\
The second term on the right-hand side in \eqref{eq: equ for initial data} involving the gradients may again be estimated in the limit by employing an argument of weak lower-semicontinuity.\\
For the other term, note that due to boundedness of $(\omega(u^n))_{n\in\N}$ in $C_b(\R;(L^1 \cap L^p)(\R^2))$ and $1 < \frac{p}{p-1} \leq p$, we may assume that $(\omega(u^n)|\omega(u^n)|^{p-2})_{n\in\N}$ converges weakly-* in $L^\infty(\R;L^{\frac{p}{p-1}}(\R^2))$ to some $\Psi$. We can then conclude the argument as in the previous case if we can prove that 
\begin{equation}\label{eq: weak convergence product}
\Psi = \omega(u)|\omega(u)|^{p-2}.
\end{equation}
Due to strong convergence in $C_{loc}(\R;L^2(\R^2))$, interpolating between $L^1(\R^2)$ and $L^2(\R^2)$ or $L^2(\R^2)$ and $L^p(\R^2)$, depending on whether $p-1 < 2$ or $p-1 \geq 2$, yields
\[|\omega(u^n)|^{p-2} \to |\omega(u)|^{p-2}\,(n \to \infty)\text{ in } L^\infty_{loc}(\R;L^{\frac{p-1}{p-2}}(\R^2)).\]
As we may also assume $\omega(u^n) \overset{\ast}{\rightharpoonup} \omega(u)\,(n \to \infty)$ in $L^\infty(\R;L^{p-1}(\R^2))$ and $\frac{1}{p-1} + \frac{p-2}{p-1} = 1$, we can conclude \eqref{eq: weak convergence product}.

$1 < p < 2$: In this case, we argue that we are actually in the situation of the case $p=2$, i.e., we have strong convergence $\|\omega(u^n)(0)\|_{L^2} \to \|\omega(u)(0)\|_{L^2}\,(n\to\infty)$, which along with a uniform $L^1$ bound implies $\|\omega(u^n)(0)\|_{L^p} \to \|\omega(u)(0)\|_{L^p}\,(n\to\infty)$ as desired.\\
As seen in \Cref{lem: NSE bdd l2 by l1}, $\omega(u^n)(t)$ is instantaneously bounded in $L^2(\R^2)$ for $t$ bounded away from the initial time $-t^n$. Therefore, assuming $t^n > 1$ for all $n \in \N$ and simply using the sequence $\tilde{t}^n \coloneqq t^n -1$, $n \in \N$, instead of $(t^n)_{n\in\N}$ in \eqref{eq: equ for initial data} and the arguments to follow, we obtain as in the case $p = 2$ that $\|\omega(u^n)(0)\|_{L^2} \to \|\omega(u)(0)\|_{L^2}\,(n\to\infty)$ as desired.
\end{proof}

We note that the arguments in the proof of \Cref{prop: attractor lp vort} also apply to inviscid limit sequences for $p \geq 2$ as the term explicitly involving the viscosity and gradient in \eqref{eq: equ for initial data} can simply be dropped when considering the limit superior and since we did not rely on bounds depending on the viscosity $\nu$ for the other terms. The essential property is that it is a priori known that the limit satisfies \eqref{eq: EE energy equ} which in the inviscid limit case would be guaranteed by \Cref{thm: weak convergence in lp}. We state this observation in the following proposition.\\
We also point out that the employed estimates that are independent of $\nu > 0$ and subsequently hold for inviscid limit sequences no longer suffice if we just assume $f \in \CalE_1 \cap \CalE_p$ for $1 < p < 2$, which is the primary reason why we make the assumption $f \in \CalE_1 \cap \CalE_{\max\lbrace 2,p\rbrace}$. If we were only to consider the damped Navier-Stokes equations, this would not be necessary.\\
Finally, we mention that the arguments in the proof of \Cref{prop: attractor lp vort} actually showed how to improve weak convergence (of the initial data) to strong convergence. An adaptation of these arguments can actually be used to prove \Cref{thm: strong convergence in lp} for $p \geq 2$, as was pointed out in the introduction of \cite{NLSW21} with reference to Remark 2 in \cite{CDE22}.

\begin{prop}\label{prop: weak to strong conv}
We assume that $p \geq 2$. Let $(\nu^k)_{k\in\N} \subset (0,1)$ be a sequence converging to $0$ and let $(f^{\nu^k})_{k\in\N}$ be a bounded sequence in $\CalE_1 \cap \CalE_p$. Suppose that there exists $f \in \CalE_1 \cap \CalE_p$ s.t. for $g^{\nu^k} \coloneqq \curl f^{\nu^k}$ and $g \coloneqq \curl f$
\[g^{\nu^k} \to g \,(k \to \infty) \text{ in } L^p(\R^2).\]
Consider a sequence $(u^{\nu^k})_{k\in\N}$ of complete weak solutions of the damped Navier-Stokes equations with right-hand sides $(f^{\nu^k})_{k\in\N}$, bounded in $\CalE_1 \cap \CalE_p$ with $\nu^k$ being the viscosity parameter associated to $u^{\nu^k}$ for every $k \in \N$. After passing to a subsequence, there exists $u \in L^\infty([0,\infty);\CalE_1 \cap \CalE_p)$ whose vorticity $\omega(u)$ satisfies the damped Euler equations with right-hand side $g$ in the renormalized sense and
\[\omega(u^{\nu^k}) \to \omega(u)\,(k\to\infty)\text{ in } C_{loc}([0,\infty);L^p(\R^2)).\]
\end{prop}

\Cref{prop: weak to strong conv} is going to be useful in combination with the following lemma.

\begin{lem}\label{lem: l2 bdd attractor}
Let $1 < p < \infty$ but assume \eqref{eq: g} for $f$. Then, $\CalA_p \subset \CalA_2$ and $\CalK_p \subset \CalK_2$ so that in particular every $u \in \CalK_p$ satisfies
\begin{equation}\label{eq: bd ancient trajectories}
\sup_{t \in \R} \|\omega(u)(t)\|_{L^2(\R^2)} \leq \frac{\|g\|_{L^2}}{\gamma}.
\end{equation}
\end{lem}

\begin{proof}
Let $B$ be a bounded subset of $\CalE_1 \cap \CalE_p$. Due to \Cref{lem: NSE bdd l2 by l1}, $S(t)B$ is bounded in $L^2(\R^2)$ for positive $t$ bounded away from $0$. Due to assumption \eqref{eq: g}, we are essentially in the situation $p = 2$ and $\CalA_2$ attracts $B$ for $\lbrace S(t) \rbrace_{t \geq 0}$ w.r.t. $\|\cdot\|_{\CalE_2}$. If we particularly consider the set $B = \CalA_p$, attraction of $\CalA_2$ and invariance of $\CalA_p$ w.r.t. $\lbrace S(t) \rbrace_{t \geq 0}$ already imply
\[\CalA_p \subset \CalA_2.\]
Using the definition of $\CalA_p, \CalA_2$ and estimate \eqref{eq: NSE vort lp bound}, one also obtains $\CalK_p \subset \CalK_2$ from this.
\end{proof}

\section{Long time average invariant measures}\label{sec: time-average invar meas}

In this section, we are concerned with the construction of invariant measures from distributions of initial data via long time averages.\\
The main theorem of this section, \Cref{thm: krylov-bogolioubov}, is essentially an application of Theorem 7 in \cite{LRR11}. We remark, however, that the stated assumptions are note quite satisfied here. The major discrepancy comes again from working on $\R^2$, where $\CalE_1 \not\subset \CalE_p$: We showed in \Cref{prop: attractor lp vort} that $\CalA_p$ is compact and attracting w.r.t. $\|\cdot\|_{\CalE_p}$ instead of $\|\cdot\|_{\CalE_1 \cap \CalE_p}$, while we would like to regard the solution semigroup of the damped Navier-Stokes equations $\lbrace S(t) \rbrace_{t \geq 0}$ as a semigroup on $\CalE_1 \cap \CalE_p$.\\
Nevertheless, the proof of Theorem 7 in \cite{LRR11} still applies and we give it here for the convenience of the reader.

\begin{thm}\label{thm: krylov-bogolioubov}
Let $\mu_0$ be a Borel probability measure on $\CalE_1 \cap \CalE_p$. Then, for any sequence $t^j \to \infty\,(j \to \infty)$, there exists a subsequence $t^{j^k} \to \infty\,(k \to \infty)$ and a Borel probability measure $\mu$ on $\CalE_p$, concentrated on $\CalA_p$, which is invariant w.r.t. $\lbrace S(t) \rbrace_{t \geq 0}$ (when restricted to $\CalE_1 \cap \CalE_p$) s.t. for any $\varphi \in C(\CalE_p)$,
\begin{equation}\label{eq: krylov-bogolioubov}
\lim_{k \to \infty} \frac{1}{t^{j^k}} \int_0^{t^{j^k}}\int_{\CalE_1 \cap \CalE_p} \varphi(S(t)u_0)\,d\mu_0(u_0)\,dt = \int_{\CalA_p} \varphi(u)\,d\mu(u).
\end{equation} 
\end{thm}

\begin{proof}
We begin the proof by making some observations on measurability.\\
First of all, we note that $\CalE_1 \cap \CalE_p$ is a Borel subset of $\CalE_p$ since it can be written as
\[\CalE_1 \cap \CalE_p = \bigcup_{l \in \N} \bigcap_{m\in\N} \lbrace u \in \CalE_p : \|\omega(u)\|_{L^1(B_m)} \leq l\rbrace,\]
which is a countable union of countable intersections of closed sets in $\CalE_p$. Moreover, on $\CalE_1 \cap \CalE_p$, the Borel-$\sigma$-algebras generated by the standard norm $\|\cdot\|_{\CalE_1 \cap \CalE_p}$ on that space and the subspace norm $\|\cdot\|_{\CalE_p}$ coincide: Since $\|\cdot\|_{\CalE_p} \leq \|\cdot\|_{\CalE_1 \cap \CalE_p}$, open sets in $\CalE_1 \cap \CalE_p$ w.r.t. $\|\cdot\|_{\CalE_p}$ are also open w.r.t. $\|\cdot\|_{\CalE_1 \cap \CalE_p}$.\\
Conversely, since $(\CalE_1 \cap \CalE_p,\|\cdot\|_{\CalE_1 \cap \CalE_p})$ is separable, it suffices to note that an arbitrary open ball $B = \lbrace v \in \CalE_1 \cap \CalE_p : \|u-v\|_{\CalE_1 \cap \CalE_p} < r \rbrace$ with center $u \in \CalE_1 \cap \CalE_p$ and radius $r > 0$ is Borel-measurable w.r.t. the Borel-$\sigma$-algebra on $\CalE_1 \cap \CalE_p$, generated by $\|\cdot\|_{\CalE_p}$. This holds since
\[B = \bigcup_{l \in \N} \bigcap_{m\in\N} \lbrace v \in \CalE_1 \cap \CalE_p: \|u-v\|_{\CalE_p} + \|\omega(u)- \omega(v)\|_{L^1(B_m)} \leq r - \frac{1}{l}\rbrace,\]
which is a countable union of countable intersections of subsets of $\CalE_1 \cap \CalE_p$ which are (relatively) closed w.r.t. $\|\cdot\|_{\CalE_p}$.\\
We now define the average measure $\overline{\mu}^{j}$ for every $j \in \N$ as the Borel probability measure on $\CalE_1 \cap \CalE_p$ given by
\[\overline{\mu}^{j}(A) \coloneqq \frac{1}{t^j}\int_0^{t^j} \mu_0(S(t)^{-1}A)\,dt\]
for every Borel measurable $A$ in $\CalE_1 \cap \CalE_p$. This is well-defined as $\lbrace S(t) \rbrace_{t \geq 0}$ is a continuous semigroup on $\CalE_1 \cap \CalE_p$.\\
Due to the above considerations, we may also view $(\overline{\mu}^j)_{j \in \N}$ as a measure on the Borel-$\sigma$-algebra of $\CalE_p$ via the formula $\overline{\mu}_j = \overline{\mu}_j(\cdot \cap (\CalE_1 \cap \CalE_p))$. We then argue that the sequence of measures $(\overline{\mu}^j)_{j \in \N}$ on $\CalE_p$ is asymptotically tight, which, by Prokhorov's theorem (cf. Theorem 1.3.9 in \cite{VW96}), implies \eqref{eq: krylov-bogolioubov} for $\varphi \in C_b(\CalE_p)$.\\
We first note that $\mu_0$, as a finite Borel measure on the separable Banach space $\CalE_1 \cap \CalE_p$, is automatically tight so that there exist compact sets $(K^n)_{n\in\N} \subset \CalE_1 \cap \CalE_p$ satisfying
\[\mu_0(K^n) \geq 1 - \frac{1}{n}\]
for every $n \in \N$. Fix $\delta > 0$ for the moment and consider the open $\delta$-enlargement
\[\CalA_{p,\delta} \coloneqq \lbrace u \in \CalE_p : \dist_{\CalE_p}(u,\CalA_p) < \delta\rbrace.\]
Due to the attraction property of $\CalA_p$ proved in \Cref{prop: attractor lp vort}, for every $n \in\N$ there exists $t_0^n = t_0^n(\delta)$ s.t. for every $t \geq t_0^n$,
\[S(t)K^n \subset \CalA_{p,\delta}.\]
This particularly implies for $t \geq t_0^n$,
\[S(t)^{-1}(\CalA_{p,\delta} \cap (\CalE_1 \cap \CalE_p)) \supset S(t)^{-1}(S(t)K^n \cap (\CalE_1 \cap \CalE_p)) = S(t)^{-1}(S(t)K^n) \supset K^n.\]
Hence, $\mu_0(S(t)^{-1}(\CalA_{p,\delta}\cap(\CalE_1 \cap \CalE_p))) \geq 1 - \frac{1}{n}$ for $t \geq t_0^n$. Therefore,
\begin{align}
\overline{\mu}^j(\CalA_{p,\delta}) &= \overline{\mu}^j(\CalA_{p,\delta} \cap (\CalE_1 \cap \CalE_p))\\
&= \frac{1}{t^j}\left(\int_0^{t^n_0} \mu_0(S(t)^{-1}(\CalA_{p,\delta} \cap (\CalE_1 \cap \CalE_p)))\,dt + \int_{t^n_0}^{t^j} \mu_0(S(t)^{-1}(\CalA_{p,\delta} \cap (\CalE_1 \cap \CalE_p)))\,dt\right)\\
&\geq \frac{t^j-t_0^n}{t^j}\left(1 - \frac{1}{n}\right),
\end{align}
which implies asymptotic tightness.\\
The Borel probablity measure $\mu$ on $\CalE_p$, obtained as weak limit of an adequate subsequence $(\overline{\mu^{j^k}})_{k\in\N}$ is clearly concentrated on $\CalE_1 \cap \CalE_p$ since every measure $\overline{\mu}^{j^k}$ is. Invariance w.r.t. $\lbrace S(t) \rbrace_{t \geq 0}$ can be seen from \eqref{eq: krylov-bogolioubov}.\\
We also obtain from the above considerations that for every $\delta > 0$,
\[\mu(\cl_{\CalE_p}(A_{p,\delta})) \geq \limsup_{k \to \infty} \overline{\mu}^{j^k}(\cl_{\CalE_p}(A_{p,\delta})) \geq 1 - \frac{1}{n},\]
where $\cl_{\CalE_p}$ denotes the closure in $\CalE_p$.
Since $n \in\N$ can be chosen arbitrarily, this implies $\mu(\cl_{\CalE_p}(A_{p,\delta})) = 1$. Due to $\bigcap_{\delta > 0} \cl_{\CalE_p}(A_{p,\delta})= \CalA_p$, we conclude $\mu(\CalA_p) = 1$.\\
The fact that in \eqref{eq: krylov-bogolioubov} we may consider all continuous functions and not just bounded continuous functions follows from $\CalA_p$ being compact in $\CalE_p$ and can be seen as in Theorem 4.1 in \cite{CZ19}.
\end{proof}

\section{Vanishing of long time average $p$-enstrophy dissipation rate}\label{sec: vanishing p-enstrophy dissip rate}
In this concluding part of the article, we prepare and prove \Cref{thm: main thm}.\\
We begin by remarking first that in \eqref{eq: main thm finite p moment}, there does not have to be a uniform bound in $\nu > 0$ for the moments on the left-hand side. This is essentially true since in \eqref{eq: vanishing mean enstrophy}, we first consider the long time limit after which all arguments to follow `` take place '' on the compact global attractor, independently from the initial data or distribution.\\
\hfill\\
As we consider the inviscid limit $(\nu \to 0)$ in this section, from now on, we will indicate the dependence of the solutions, the attractors, etc. on the viscosity parameter $\nu > 0$ by adding it as a superscript.\\
Let us now describe the setting in which \Cref{thm: main thm} holds. We suppose $1 < p < \infty$, let $r \coloneqq \max\lbrace 2,p \rbrace$ and consider functions $(f^\nu)_{\nu > 0}$ that are bounded in $\CalE_1 \cap \CalE_r$. We denote $g^\nu \coloneqq \curl f^\nu$ and suppose that for some $f \in \CalE_1 \cap \CalE_p$ with $g \coloneqq \curl f$, 
\[g^\nu \to g \,(\nu \to 0)\text{ in } L^r(\R^2).\]
Then we may apply all results from the previous sections. The solution operator to the Cauchy problem will be denoted by $\Sigma^\nu\colon \CalE_1 \cap \CalE_p \to C_{loc}([0,\infty);\CalE_1 \cap \CalE_p)$ and the solution semigroup on $\CalE_1 \cap \CalE_p$ of the damped Navier-Stokes equations by $\lbrace S^\nu(t)\rbrace_{t \geq 0}$.\\
The vorticity may be interpreted as a continuous mapping from $\CalE_p$ to $L^p(\R^2)$ and we denote the composition of $\Sigma^\nu$ with the vorticity mapping by
\[\Sigma^\nu_\omega \colon \CalE_1 \cap \CalE_p \to C_{loc}([0,\infty);L^p(\R^2)).\]

Since we do not have a well-defined solution semigroup for the damped Euler equations, we work on the level of trajectories in the next lemma and consider the semigroup of time-shifts $\lbrace T(t) \rbrace_{t \geq 0}$ on $C_{loc}([0,\infty);L^p(\R^2))$. 

\begin{lem}\label{lem: stationary enstrophy balance}
Let $\rho$ be a Borel probability measure on $\CalF \coloneqq C_{loc}([0,\infty);L^p(\R^2))$ with support in the set of all renormalized solutions of the vorticity formulation of the damped Euler equations with right-hand side $g$ as in \Cref{thm: EE existence}, invariant w.r.t. $\lbrace T(t) \rbrace_{t \geq 0}$ and satisfying
\begin{equation}\label{eq: finite p-moment}
\int_{\CalF}\int_{\R^2} |\omega(0)|^p\,dx\,d\rho(\omega) < \infty.
\end{equation}
Then
\begin{equation}\label{eq: stationary enstrophy balance}
\int_{\CalF} \int_{\R^2} -\gamma |\omega(t)|^p + g\omega(t)|\omega(t)|^{p-2}\,dx\,d\rho(\omega) = 0
\end{equation}
for every $t \geq 0$.
\end{lem}

\begin{proof}
Due to the invariance of $\rho$ w.r.t. $\lbrace T(t) \rbrace_{t \geq 0}$, \eqref{eq: finite p-moment} holds for every $t \geq 0$ and we have 
\begin{equation}
\begin{split}
&\int_{\CalF} \int_{\R^2} -\gamma |\omega(t)|^p + g\omega(t)|\omega(t)|^{p-2}\,dx\,d\rho(\omega)\\
&\qquad= \int_{\CalF}\int_{\R^2} -\gamma |\omega(0)|^p + g\omega(0)|\omega(0)|^{p-2}\,dx\,d\rho(\omega)
\end{split}
\end{equation}
so that the left-hand side is constant in $t$. Then, by taking the time average and applying both Fubini's theorem and \eqref{eq: EE energy equ}, we obtain for all $0 \leq t \leq M$,
\begin{equation}
\begin{split}
&\int_{\CalF} \int_{\R^2} -\gamma |\omega(t)|^p + g\omega(t)|\omega(t)|^{p-2}\,dx\,d\rho(\omega)\\
= &\frac{1}{M}\int_0^M\int_{\CalF} \int_{\R^2} -\gamma |\omega(t)|^p + g\omega(t)|\omega(t)|^{p-2}\,dx\,d\rho(\omega)\,dt\\
= &\frac{1}{M}\int_{\CalF}\int_0^M \int_{\R^2} -\gamma |\omega(t)|^p + g\omega(t)|\omega(t)|^{p-2}\,dx\,dt\,d\rho(\omega)\\
= &\frac{1}{M}\int_{\CalF} (\|\omega(M)\|_{L^p}^p - \|\omega(0)\|_{L^p}^p)\,d\rho(\omega),
\end{split}
\end{equation}
which is $0$ due to invariance of $\rho$ w.r.t. $\lbrace T(t) \rbrace_{t \geq 0}$.
\end{proof}

In preparation for the proof of \Cref{thm: main thm}, for every $\nu > 0$, we define the set of vorticity of bounded, complete solutions with right-hand side $g^\nu$ and restricted to non-negative times
\begin{equation}\label{eq: complete vort sols}
\CalV^\nu_p \coloneqq \lbrace \omega(u)|_{[0,\infty)} \in C_{loc}([0,\infty);(L^1 \cap L^p)(\R^2)) : u \in \CalK^\nu_p\rbrace.
\end{equation}
Moreover, let $\CalV^0_p$ be the set of all renormalized solutions of the vorticity formulation of the damped Euler equations with right-hand side $g$ as in \Cref{thm: EE existence}. Due to \Cref{lem: l2 bdd attractor} and \Cref{prop: weak to strong conv}, the following lemma holds.

\begin{lem}\label{lem: U is precompact}
For any sequence $\nu^k \to 0\,(k\to\infty)$, the closure in $C_{loc}([0,\infty);L^p(\R^2))$ of the union 
\begin{equation}\label{eq: union NSE and EE}
\CalV_p \coloneqq \bigcup_{k \in \N} \CalV^{\nu^k}_p \cup \CalV^0_p
\end{equation}
is compact.
\end{lem}

We may proceed to prove \Cref{thm: main thm} by a similar approach as \cite{CR07}, but with the simplifications described in the introduction to this article and prepared in the previous sections.

\begin{proof}[Proof of \Cref{thm: main thm}]
Suppose that \eqref{eq: vanishing mean enstrophy} is false. Then there exists a sequence $(\nu^n)_{n\in\N} \subset (0,1)$ converging to 0 and some $\delta > 0$ s.t. for every $n \in \N$, there exists a sequence of times $(t^j = t^j(n))_{j\in\N}$ with $t^j \to \infty \,(j \to \infty)$ and
\[\nu^n \frac{1}{t^j}\int_0^{t^j}\int_{\CalE_1 \cap \CalE_p}\int_{\R^2} |\nabla|\omega(S^{\nu^n}(t) u_0)|^\frac{p}{2}|^2\,dx\,d\mu_0^{\nu^n}(u_0)\,dt \geq \delta,\]
$n,j \in \N$. By \eqref{eq: NSE energy equ} and Fubini's theorem, the left-hand side is bounded by 
\begin{equation}
\begin{split}
\frac{p^2}{4(p-1)}\frac{1}{t^j}\int_{\CalE_1 \cap \CalE_p}\int_0^{t^j}\int_{\R^2} -\gamma |\omega(S^{\nu^n}(t) u_0)|^p + g^{\nu^n}\omega(S^{\nu^n}(t) u)|\omega(S^{\nu^n}(t) u)|^{p-2}\,dx\,dt\,d\mu_0^{\nu^n}(u_0)\\
\quad + \frac{1}{t^j}\int_{\CalE_1 \cap \CalE_p}\frac{p}{4(p-1)}(\|\omega(u_0)\|_{L^p}^p - \|\omega(S^{\nu^n}(t^j) u_0)\|_{L^p}^p)\,d\mu_0^{\nu^n}(u_0).
\end{split}
\end{equation}
We have
\begin{equation}
\begin{split}
&\frac{1}{t^j}\int_{\CalE_1 \cap \CalE_p}\frac{p}{4(p-1)}(\|\omega(u_0)\|_{L^p}^p - \|\omega(S^{\nu^n}(t) u)\|_{L^p}^p)\,d\mu_0^{\nu^n}(u_0)\\
\leq &\frac{C}{t^j}\int_{\CalE_1 \cap \CalE_p}\frac{p}{4(p-1)}(\|\omega(u_0)\|_{L^p}^p + \frac{\|g^{\nu^n}\|_{L^p}^p}{\gamma^p})\,d\mu^{\nu^n}_0(u_0)\\
\to &0\, (j \to \infty)
\end{split}
\end{equation}
due to \eqref{eq: main thm finite p moment}.
The integrals 
\[u \mapsto \int_{\R^2}-\gamma|\omega(u)|^p\,dx\text{ and } u \mapsto \int_{\R^2}g^{\nu^n} \omega(u)|\omega(u)|^{p-2}\,dx\]
may be seen as continuous mappings on $\CalE_p$. Therefore, by \Cref{thm: krylov-bogolioubov}, after passing to a subsequence $t^{j^k} \to \infty$, there exists an invariant Borel probability measure $\mu^{\nu^n}$ on $\CalE_p$, concentrated on $\CalA_p^{\nu^n}$ so that 
\begin{equation}\label{eq: proof main thm 1}
\begin{split}
&\frac{4(p-1)}{p^2}\delta\\
\leq &\lim_{k \to \infty} \frac{1}{t^{j^k}} \int_0^{t^{j^k}}\int_{\CalE_1 \cap \CalE_p}\int_{\R^2} -\gamma|\omega(S^{\nu^n}(t) u_0)|^p + g^{\nu^n}\omega(S^{\nu^n}(t) u_0)|\omega(S^{\nu^n}(t) u_0)|^{p-2}\,dx\,d\mu_0^{\nu^n}(u_0)\,dt\\
= &\int_{\CalA_p^{\nu^n}} \int_{\R^2} -\gamma|\omega(u)|^p + g^{\nu^n}\omega(u)|\omega(u)|^{p-2}\,dx\,d\mu^{\nu^n}(u).
\end{split}
\end{equation}
Now we define for every $n\in\N$ the pushforward measures 
\[\rho^{\nu^n} \coloneqq {\Sigma_\omega^{\nu^n}}\sharp \mu^{\nu^n}\]
on $\CalV^{\nu^n}_p \subset \CalV_p \subset C_{loc}([0,\infty);L^p(\R^2))$, with $\CalV_p$ as in \eqref{eq: union NSE and EE}. We fix some $\tau \geq 0$ s.t. due to invariance of $\mu^{\nu^n}$ w.r.t. $\lbrace S^{\nu^n}(t) \rbrace_{t \geq 0}$, the right-hand side in \eqref{eq: proof main thm 1} is equal to
\begin{equation}\label{eq: proof main thm 2}
\int_{\CalV_p} \int_{\R^2} -\gamma |\omega(\tau)|^p + g^{\nu^n}\omega(\tau)|\omega(\tau)|^{p-2}\,dx\,d\rho^{\nu^n}(\omega).
\end{equation}
We can view every $\rho^{\nu^n}$ as a measure on the closure $\overline{\CalV_p}$ of $\CalV_p$ in $C_{loc}([0,\infty);L^p(\R^2))$, which is compact by \Cref{lem: U is precompact}. Then, a subsequence of $(\rho^{\nu^n})_{n \in \N}$ converges weakly-* to a Borel probability measure $\rho$, i.e.,
\[\int_{\overline{\CalV_p}} \Psi(u)\,d\rho^{\nu^n}(u) \to \int_{\overline{\CalV_p}} \Psi(u)\,d\rho(u)\,(n \to \infty)\]
for every continuous $\Psi\colon\overline{\CalV_p} \to \R$. Passing to the limit $(n \to \infty)$ in \eqref{eq: proof main thm 1} then yields
\begin{equation}\label{eq: proof main thm 3}
\frac{4(p-1)}{p^2}\delta \leq \int_{\overline{\CalV_p}} \int_{\R^2}-\gamma|\omega(\tau)|^p + g\omega(\tau)|\omega(\tau)|^{p-2}\,dx\,d\rho(\omega).
\end{equation}
We now derive the contradiction from this. For every $\omega \in \supp\rho$, there exists a sequence $(\tilde{u}^{\nu^n})_{n\in\N} \subset \CalK_p^{\nu^n}$ s.t. (after restricting the time to $[0,\infty)$) $\omega(\tilde{u}^{\nu^n}) \in \supp\rho^{\nu^n} \subset \CalV_p^{\nu^n}$ for every $n \in \N$ and 
\[\omega(\tilde{u}^{\nu^n}) \to \omega\,(n\to\infty) \text{ in }C_{loc}([0,\infty);L^p(\R^2)),\]
cf. \cite{WW23}, Lemma 2.20. In particular, $\omega(\tilde{u}^{\nu^n}(0)) \to \omega(0)\,(n\to\infty)$ in $L^p(\R^2)$ so that by \Cref{thm: strong convergence in lp}, $\omega \in \CalV_p^0$. This means that $\supp\rho \subset \CalV_p^0$ and we may apply \Cref{lem: stationary enstrophy balance}. Note that \eqref{eq: finite p-moment} is satisfied as $\rho$ is concentrated on $\overline{\CalV_p}$, which is bounded in $C_b([0,\infty);L^p(\R^2))$. But then \eqref{eq: stationary enstrophy balance} contradicts \eqref{eq: proof main thm 3}.
\end{proof}

\section*{Acknowledgments}
The author would like to thank Emil Wiedemann for helpful discussions and comments on the manuscript.

\nocite{*}
\bibliography{Long_time_enstrophy_dissipation_rate.bib}
\bibliographystyle{abbrv}

\end{document}